\documentclass[reqno]{amsart}
\usepackage{amsfonts,amsmath,amsthm,amssymb,latexsym, cite}%
\usepackage{graphicx,verbatim}
\usepackage[colorlinks,plainpages, citecolor=red,
bookmarksnumbered,linkcolor=blue]{hyperref}

\theoremstyle{plain}
\newtheorem{thm}{Theorem}
\newtheorem{defn}{Definition}

\newtheorem{lem}{Lemma}

\numberwithin{equation}{section}

\renewcommand{\Im}{\mathop{\rm Im}}
\newcommand{\supp}{\mathop{\rm supp}}

\renewcommand{\kappa}{\varkappa}
\newcommand{\rmi}{i}
\newcommand{\Real}{\mathbb R}
\newcommand{\eps}{\varepsilon}
\newcommand{\en}{{\eps\nu}}

\newcommand{\cI}{\mathcal{I}}
\newcommand{\cT}{\mathcal{T}}

\newcommand{\cH}{\mathcal{H}}
\newcommand{\OprT}{\mathcal{T}_\eps}
\newcommand{\OprH}{\mathcal{H}_\en}
\newcommand{\OprHe}{\mathcal{H}_{\eps\hskip0.4pt\nu_\eps}}
\newcommand{\dom}{\mathop{\rm dom}}
\newcommand{\ra}{\rangle}
\newcommand{\la}{\langle}
\newcommand{\fpr}[1]{{#1}^{(-1)}}
\newcommand{\spr}[1]{{#1}^{(-2)}}
\newcommand\xe{\left(\tfrac x\eps\right)}

\newcommand\se{\left(\tfrac s\eps\right)}

\begin{document}

\title[Some remarks on 1D Schr\"{o}dinger operators]
{Some remarks on 1D Schr\"{o}dinger operators with localized magnetic and electric potentials}

\author{Yuriy Golovaty}%

\address{Department of Differential Equations,
  Ivan Franko National University of Lviv\\
  1 Universytetska str., 79000 Lviv, Ukraine}


\begin{abstract}
One-dimensional Schr\"{o}dinger operators with singular perturbed magnetic and electric potentials are considered. We study the strong resolvent convergence of two families of the operators with potentials shrinking to a point. Localized $\delta$-like  magnetic fields are combined with $\delta'$-like perturbations of the electric potentials as well as localized rank-two perturbations. The limit results obtained heavily depend on zero-energy resonances  of the electric potentials. In particular, the approximation for a wide class of point interactions in one dimension is obtained.
\end{abstract}

\maketitle

\section{Introduction}
The present paper is concerned with the  convergence of families of singularly perturbed one-dimensional magnetic Schr\"{o}dinger operators.
Our motivation of the study on this convergence comes from an application to the scattering of quantum particles by sharply localized  potentials and finite rank perturbations. The main purpose is to construct solvable models in terms of the point interactions describing with admissible fidelity the real quantum interactions.
The Schr\"{o}dinger operators with  potentials that are distributions supported on  discrete sets (such potentials are usually termed point interactions) have attracted consi\-de\-rable attention since the $1980$s.
It is an extensive subject with a large literature (see \cite{AlbeverioGesztesyHoeghKrohnHolden2edition, AlbeverioKurasov}, and the references given therein).

It is well-known that all nontrivial point interactions at a point $x$ can be described by the coupling conditions
\begin{equation}\label{PIConditions}
\begin{pmatrix} \psi(x+0) \\ \psi'(x+0) \end{pmatrix}
=e^{i\varphi}\begin{pmatrix} c_{11} & c_{12} \\ c_{21} & c_{22} \end{pmatrix}
\begin{pmatrix} \psi(x-0) \\ \psi'(x-0)\end{pmatrix},
\end{equation}
where $\varphi\in [-\frac{\pi}{2},\frac{\pi}{2}]$, $c_{kl}\in\mathbb{R}$, and $c_{11}c_{22}-c_{12}c_{21}=1$ (see, e.g., \cite{AlbeverioDabrowskiKurasov:1998, ChernoffHughes:1993}). The nontrivia\-li\-ty of point interactions means that  the associated self-adjoint operator cannot be presented as a direct sum of two operators acting in $L_2( -\infty, 0)$ and $L_2(0, \infty)$.
For the quantum systems described by the Schr\"{o}dinger operators with regular potentials localized in a neighbourhood of $x$ one can often assign the Schr\"{o}dinger operators with the point interactions \eqref{PIConditions} so that the corresponding zero-range models govern the  quantum dynamics of the true interactions with adequate accuracy, especially for the low-energy particles.
In this context, the inverse problem is also of in\-te\-rest. The important question is  how to approximate a given point interaction by  Schr\"{o}dinger operators with  localized regular potentials or  finite-rank perturbations.

We study the families of the Schr\"{o}dinger operators that can be partially viewed  as  regularizations of the pseudo-Hamiltonians
\begin{equation}\label{HeuiristicOprs}
\begin{gathered}
    \left(\rmi\,\frac{d}{dx}+a \delta(x)\right)^2+b \delta'(x)+c\delta(x), \\
    \left(\rmi\,\frac{d}{dx}+a \delta(x)\right)^2+b \big(\langle \delta'(x),\,\cdot\,\rangle \,\delta(x)
  +\langle \delta(x),\,\cdot\,\rangle\, \delta'(x)\big)+c\delta(x),
\end{gathered}
\end{equation}
where $\delta$ is Dirac's delta function. We note that
$\delta'(x) y=y(0)\delta'(x)-y'(0)\delta(x)$ for continuously differentiable functions $y$  at the origin. Thus we may formally
regard the $\delta'$ potential as rank-two perturbation
$\delta'(x)y= \langle \delta(x),y\rangle\, \delta'(x)+\langle \delta'(x),y\rangle\, \delta(x)$.
However, both the heuristic operators have  generally no ma\-the\-ma\-ti\-cal meaning.
So it is not surprising that  different  re\-gu\-larizations of the distributions in \eqref{HeuiristicOprs}  lead to different  self-adjoint operators in the limit. Therefore the pseudo-Hamiltonians \eqref{HeuiristicOprs} can be regarded   as  a symbolic notation only for a wide variety of quantum systems with  quite different properties depending on the shape of the short-range potentials.

Recently a class of the Schr\"{o}dinger operators with
piece-wise constant $\delta'$-poten\-tials were studied by Zolotaryuk a.o. \cite{Zolotaryuk:2008, Zolotaryuk:2010, Zolotaryuks11, ZolotaryukThreeDelta17}; the resonances in the transmission probability for the scattering problem were established. As was shown in \cite{GolovatyManko:2009, GolovatyHryniv:2010, GolovatyHryniv:2013, MankoJPA:2010} these resonances deal with the existence of zero-energy resonances and the half-bound states for singular localized potentials. The zero-energy resonances have a profound effect on the limiting behaviour of the Schr\"{o}dinger operators with $\delta'$-poten\-tials. Such operators also arose in connection with the appro\-xi\-mation of  smooth planar quantum waveguides by  quantum graph \cite{AlbeverioCacciapuotiFinco:2007,CacciapuotiExner:2007, CacciapuotiFinco:2007}; a similar resonance phenomenon  was  obtained. The reader also interested in the literature on other aspects of $\delta'$-potentials and  $\delta'$-interactions as well as approximations of point interactions by local and  non-local perturbations is referred to \cite{AlbeverioNizhnik2013, AlbeverioFassariRinaldi2013, AlbeverioFassariRinaldi2015, ExnerManko2014, GadellaNegroNietoPL2009, GadellaGlasserNieto2011, GadellaGarcia-Ferrero2014}.

It is known that  one dimensional Schr\"{o}dinger operators
\begin{equation*}
    H(b)=\left(\rmi\frac{d}{dx}+b(x)\right)^2+V(x)
\end{equation*}
with continuous magnetic potentials  are not especially interesting, because any continuous  field $b$ is equivalent under a smooth gauge transformation to $0$. This means that the operator $H(b)$ with a continuous gauge field  is unitarily equivalent to the Schr\"{o}dinger operator $H(0)=-\frac{d^2}{dx^2}+V(x)$ without a magnetic field.
The authors of \cite{CoutinhoNogamiTomio:1999} have even asserted that the phase parameter $\varphi$ in conditions \eqref{PIConditions} is redundant and it produces no interesting effect. They have stated that if the time-reversal invariance is imposed, the number of the parameters that specify the interactions \eqref{PIConditions} can be reduced to three.

For the case of  singular magnetic potentials, however, there are certain nontrivial examples \cite{AlbeverioFeiKurasov}, pointing out that this case is more subtle.  Albeverio, Fei and Kurasov \cite{AlbeverioFeiKurasov} have shown that the  phase parameter is  not redundant if nonstationary problems are  concerned.  The  phase parameter can  be  interpreted as  the  amplitude of a  singular gauge field.
As stated in \cite{KurasovJMAA:1996} a nonzero phase $\varphi$  in the coupling conditions \eqref{PIConditions} may appear if and only if the singular gauge field is present. However, it is noteworthy that the  factor $e^{i\varphi}$ also appeared in the solvable model for the Schr\"{o}dinger operators without a magnetic field that is perturbed by a rank-two operator \cite{GolovatyIEOT2018}. We also want to note that  Theorem~\ref{ThmT} in the present paper gives an example of an exactly solvable model in which the magnetic field has an effect on all coefficients $c_{kl}$ in   \eqref{PIConditions}, not only on factor $e^{i\varphi}$.

Another reason to study the 1-D Schr\"{o}dinger operators with  magnetic fields comes from the quantum graph theory
which is a useful tool in modelling numerous physical phenomena.
One of the fundamental questions of this theory consists of justifying the possibility of approximating dynamics of a quantum particle confined to real-world mesoscopic waveguides of small width $d$ by its dynamics on the graph obtained in the limit as $d$ vanishes. In  \cite{ExnerPost:2012} the authors demonstrated that any self-adjoint coupling in a quantum graph vertex can be approximated by a family of magnetic Schr\"{o}dinger operators on a tubular network built over the graph.

The magnetic Schr\"{o}dinger operators and the Dirac Hamiltonians  with Aharonov-Bohm fields  have been  discussed from various aspects by many authors. We confine ourself to a brief overview of the most relevant papers. For the mathematical foundation of the magnetic Schr\"{o}dinger operators we refer the reader to the paper of Avron, Herbst, and Simon \cite{AvronHerbstSimon:1978}.
In two dimension, the norm resolvent convergence of the Schr\"{o}dinger operators
\begin{equation*}
    \mathcal{H}_\eps=\left(\rmi \nabla+\eps^{-1}A(x/\eps)\right)^2+\eps^{-2}V(x/\eps)
\end{equation*}
with singularly scaled  magnetic and  electric potentials was studied by Tamura \cite{Tamura2001}. The magnetic potential had the  $\delta$-like field $\eps^{-2}b(x/\eps)=\eps^{-1}\nabla\times A(x/\eps)$, and $b$ and $V$ were smooth vector functions in $\Real^2$ of compact support.  The limit operator strongly depends on the total flux of magnetic field and on the resonance space at zero energy. The scattering by a magnetic field with small support and the convergence to the scattering amplitude by $\delta$ magnetic field were studied in \cite{Tamura1999}. In \cite{Tamura2003}, the case of relativistic particles moving in the Aharonov-Bohm magnetic field with a $\delta$-like singularity was considered.
The author approximated  the point–like field by smooth ones  and found the limit self-adjoint operators uniquely specified by physically and mathematically reasonable  boundary conditions at the origin.

The present paper can be viewed as a natural continuation of  our previous works
\cite{GolovatyMFAT:2012, GolovatyIEOT:2012, GolovatyJPA2018, GolovatyIEOT2018}, in which the case without of a magnetic field was treated.


\section{Statement of Problem and Main Results}
\label{SecStatement}

Let us consider the Schr\"{o}dinger operator
\begin{equation*}
  H_0=-\frac{d^2}{dx^2}+V_0
\end{equation*}
in $L_2(\Real)$,  where potential $V_0$ is  real-valued, measurable and locally bounded. We also assume that $V_0$ is bounded from below in $\Real$.
We turn now to our primary task of studying the limit behaviour of two families of operators in $L_2(\Real)$, which can be treated as perturbations of $H_0$.

\subsection{Hamiltonians with localized potentials}
First we consider the self-adjoint operators
\begin{equation}\label{OprHenm}
\OprH= \Big(\rmi\,\frac{d}{dx}
    +\frac{1}{\eps}\,A\Big(\frac{x}{\eps}\Big)\Big)^2
    +V_0(x)
    +\frac{\alpha}{\eps^2}\,V\Big(\frac{x}{\eps}\Big)
    +\frac{1}{\nu}\,U\Big(\frac{x}{\nu}\Big),
\end{equation}
where  $\eps$ and $\nu$ are small positive parameters, and  $\alpha$ is a real coupling constant.
Let $A$, $V$ and $U$ be  real-valued, measurable and  bounded  functions of compact support.  Suppose furthermore that $A\in \mathcal{AC}(\Real)$. The domain of $\OprH$ coincides with $\dom H_0$, because the perturbation has a compact support.
Note that we consciously equipped potential $V$ only  with a coupling constant. As we will see later, the limit behaviour of $\OprH$ crucially depends on $\alpha$.

The potentials $\alpha\eps^{-2}V(\eps^{-1}x)+\nu^{-1}U(\nu^{-1}x)$ converge, as $\eps$ and $\nu$ go to zero, to a distribution having the form $b_1\delta'(x)+b_0\delta(x)$, if $V$ has a zero-mean value, and they diverge otherwise. Hence parameter $\eps$ describes the rate of shrinking for the $\delta'$-like potential (as well as the magnetic potential), while $\nu$ is the rate of shrinking for the $\delta$-like potential.
The sequence $\eps^{-1}A(\eps^{-1}x)$ converges to  $\mu\delta(x)$ as $\eps\to 0$ in the sense of distributions, where
\begin{equation}\label{Sigma}
     \mu=\int_\Real A(x)\,dx.
\end{equation}
In the partial cases, operators $\OprH$ can be regarded as a regularization of the first pseudo-Hamiltonian in \eqref{HeuiristicOprs}.

Let us introduce some characteristics of the potentials $V$ and $U$.
\begin{defn}
 We say that the Schr\"odinger operator~$-\frac{d^2}{d x^2}+\alpha V$ in $L_2(\Real)$ possesses a \emph{zero-energy resonance}  if there exists a non trivial solution~$v_\alpha\colon \Real\to\Real$ of the equation
$ -v'' +\alpha V v= 0$
that is bounded on the whole line. We call $v_\alpha$ the \emph{half-bound state} of $\alpha V$.
\end{defn}
We will simply say that
the potential $\alpha V$ is \emph{resonant} and  it possesses a half-bound state $v_\alpha$.
Let us denote by $\mathcal{R}(V)$ the set of all coupling constants $\alpha$ for which the potential $\alpha V$ is resonant, and introduce the mapping $\theta\colon \mathcal{R}(V)\to\Real$ defined by
\begin{equation}\label{MapTheta}
  \theta(\alpha)=\frac{v_\alpha^+}{v_\alpha^-},
\end{equation}
where $v_\alpha^-=\lim\limits_{x\to-\infty}v_\alpha(x)$ and
$v_\alpha^+=\lim\limits_{x\to+\infty}v_\alpha(x)$.
Let $\Lambda=[0,+\infty]$ be the set containing the point $+\infty$.

We also define the mapping $\gamma\colon \mathcal{R}(V)\times\Lambda\to \Real$ as follows:
\begin{align}
\label{MapKappa0}
&\gamma(\alpha, 0)= \frac{v_\alpha^2(0)}{v_\alpha^-\,v_\alpha^+} \,\int_\Real U\,dt,
\\\label{MapKappa}
&\gamma(\alpha, \lambda)= \frac{1}{v_\alpha^-\,v_\alpha^+}\, \int_\Real U(t)\, v_\alpha^2(\lambda t)\,dt \qquad \text{for } \lambda\in(0,+\infty),
\\\label{MapKappaInfty}
&\gamma(\alpha, +\infty)=\theta(\alpha)\int_{\Real_+}\kern-2pt U\,dt+
    \theta(\alpha)^{-1}\int_{\Real_-}\kern-2pt U\,dt.
\end{align}
We follow the notation used in \cite{GolovatyIEOT:2012}.
This mapping describes different kinds of the resonance interactions between the potentials $\alpha V$ and $U$ in the limit.
Both the mappings $\theta$ and $\gamma$ are well defined as we will show below in Lemma~\ref{LemmaRV}.

Let us introduce the subspace $\mathcal{V}$ in $L_2(\Real)$ as follows.
We say that $h$ belongs to $\mathcal{V}$ if there exist two functions $h_-$ and $h_+$ belonging to $\dom H_0$ such that $h(x)=h_-(x)$ for $x<0$ and $h(x)=h_+(x)$ for $x>0$.

\begin{thm}\label{Thrm1}
Suppose that a sequence $\{\nu_\eps\}_{\eps>0}$ of positive numbers is such that $\nu_\eps\to 0$ and ratio $\nu_\eps/\eps$ tends to  $\lambda\in \Lambda$ as $\eps\to 0$, i.e., this ratio has a finite or infinite limit.
If $\alpha\in \mathcal{R}(V)$, then  family of operators $\OprHe$ converges in the strong resolvent sense as $\eps\to 0$ to the operator $\cH=\cH(\alpha, \lambda)$  defined by $\cH\phi=-\phi''+V_0\phi$ on functions $\phi$ in~$\mathcal{V}$ subject to the   conditions
\begin{equation}\label{PointIntsH}
\begin{pmatrix} \phi(+0) \\\phi'(+0) \end{pmatrix}
=
e^{i \mu}\begin{pmatrix}
  \theta(\alpha) &  0\\
  \gamma(\alpha,\lambda) &  \theta(\alpha)^{-1}
\end{pmatrix}
\begin{pmatrix} \phi(-0) \\ \phi'(-0) \end{pmatrix}.
\end{equation}
\end{thm}

By analogy with the results in \cite{GolovatyIEOT:2012}, if potential $\alpha V$ is not resonant, the the limit operator is the direct sum of two Dirichlet operators acting in $L_2(-\infty,0)$ and $L_2(0,+\infty)$; that is, coupling conditions \eqref{PointIntsH} must be substituted by the Dirichlet condition $\phi(0)=0$.

It is worth noting that  explicit relations \eqref{MapTheta}-\eqref{MapKappaInfty} between the matrix entries $\theta(\alpha)$, $\gamma(\alpha,\lambda)$ and  potentials $V$ and $U$ make it possible to carry out a quantitative analysis of this quantum system, e.g. to compute approximate values of the scattering data.

\subsection{Hamiltonians with localized rank-two perturbations}
We now turn our attention to another family of operators
\begin{equation}\label{Teps}
\OprT=\Big(\rmi\,\frac{d}{dx}
    +\frac{1}{\eps}\,A\Big(\frac{x}{\eps}\Big)\Big)^2
    +V_0(x)+\frac1{\eps^3}\,F_\eps+\frac1\eps\, U\Big(\frac{x}{\eps}\Big),
\end{equation}
where $F_\eps=F_\eps(f_1,f_2)$ are  rank-two operators having the form
\begin{multline}\label{OprQ}
  (F_\eps \phi)(x)=
  \bar{\beta}\,\la f_2(\eps^{-1}\,\cdot\,),\phi\ra\,f_1\xe
  +\beta\,\la f_1(\eps^{-1}\,\cdot\,),\phi\ra\,f_2\xe
  \\=\int_\Real \Big(\bar{\beta} f_1\xe\bar{f}_2\se
  +\beta\bar{f}_1\se f_2\xe\Big)\phi(s)\,ds.
\end{multline}
Here $\langle\cdot,\cdot \rangle$ is the inner scalar product $L_2(\Real)$. From now on, the  norm in $L_2(\Real)$ will be denoted by  $\|\cdot\|$.
Ope\-rators $\OprT$ can be viewed as a regularization of the second pseudo-Hamiltonian in \eqref{HeuiristicOprs}.
Assume that  $f_1$, $f_2$ and $q$ are measurable and  bounded  functions of compact support and $\beta$ is a complex coupling constant. The potential $q$ is real-valued.

Let us  also consider  rank-two perturbation of the free the Schr\"{o}dinger operator
$$
   B=-\frac{d^2}{dx^2}+\bar{\beta}\,\la h_2,\,\cdot\,\ra\,h_1+\beta\,\la h_1,\,\cdot\,\ra\,h_2,\quad \dom B=W_2^2(\Real),
$$
where $h_1$ and $h_2$ are functions of compact support.
\begin{defn}
We say that   operator~$B$ possesses a \emph{zero-energy resonance}
provided there exists a nontrivial solution of the equation
 \begin{equation}\label{HBSEqn}
 -v''+\bar{\beta}\,\la h_2, v\ra\, h_1+\beta\,\la h_1,v\ra\, h_2= 0
\end{equation}
that is bounded on the whole line. This solution is called a \emph{half-bound state} of $B$.
We also say that $B$ admits a \emph{double zero-energy resonance}, if
there exist two linearly independent half-bound states.
\end{defn}

We will denote by $\mathcal{R}(h_1,h_2)$ the set of all coupling constants $\beta$, for which operator $B$ admits a double zero-energy resonance.

Let $h^{(-1)}$ and $ h^{(-2)}$ be the first and second  antiderivatives
\begin{equation*}
    h^{(-1)}(x)=\int_{-\infty}^x h(s)\, ds, \qquad   h^{(-2)}(x)=\int_{-\infty}^x (x-s)h(s)\, ds
\end{equation*}
for  functions of compact support. Note  if $h$ has zero mean, then $h^{(-1)}$ is also a function of compact support. Also, we set
\begin{equation}\label{FuncA}
    a(x)=\int_{-\infty}^{x}   A(t)\,dt.
\end{equation}
Let us introduce notation
\begin{equation}\label{GkNkP}
 g_k=e^{-i a} f_k, \quad  n_k=\|g_k^{(-1)}\|, \quad p=\la\fpr{g_1},\fpr{g_2}\ra,
\end{equation}
provided $g_1$ and $g_2$ are functions of zero mean values. Therefore $n_k$ and $p$ are well defined, since $\fpr{g}_k$ are functions of compact support. Let
\begin{equation*}
\omega_\beta=e^{i\arg(\beta^{-1}+p)}n_2 g_1^{(-2)}-n_1 g_2^{(-2)}.
\end{equation*}
Function $\omega_\beta$ is  constant outside  some compact set containing the supports of $f_k$. Of course $\omega_\beta(x)=0$ for negative $x$ with the large absolute value. Write
\begin{equation*}
 \kappa=\lim_{x\to+\infty}\omega_\beta(x).
\end{equation*}
In the case of the double zero-energy resonance function $\omega_\beta$ is a half-bound state of $B$ with $h_k=g_k$ (see Lemma~\ref{LemmaHBS} below).
 We also set
\begin{equation*}
 a_0=\int_\Real U\,dx,\qquad a_1=\int_\Real U\,\omega_\beta\,dx,\qquad a_2=\int_\Real U\,|\omega_\beta|^2\,dx.
\end{equation*}

\begin{thm}\label{ThmT}
 Assume that  $f_1$ and $f_2$ are linearly independent, $e^{-i a} f_1$ and $ e^{-i a} f_2$ have zero means, and   $\beta\in \mathcal{R}(e^{-i a}f_1,e^{-i a}f_2)$. Suppose also that  $a_2\neq \bar{\kappa} a_1$.
 Then operator family $\OprT$  converges as $\eps\to 0$ in the strong resolvent sense to  operator $\cT$  defined by $\cT\phi=-\phi''+V_0\phi$ on functions $\phi$ in~$\mathcal{V}$ subject to the conditions
 \begin{equation}\label{CouplCondsT}
 \begin{pmatrix} \phi(+0) \\\phi'(+0) \end{pmatrix}
= e^{i\big(\mu-\arg (a_2-\overline{\kappa} a_1)\big)}
\begin{pmatrix}
\dfrac{a_0|\kappa|^2-2{\rm Re}(\overline{\kappa} a_1)+a_2}{|a_2-\overline{\kappa} a_1|}    &    \kern-10pt\phantom{\dfrac{\int}{\int}}\dfrac{|\kappa|^2}{|a_2-\overline{\kappa} a_1|}\\
      \dfrac{a_0a_2-|a_1|^2}{|a_2-\overline{\kappa} a_1|}                &      \kern-10pt\phantom{\dfrac{\int}{\int}}\dfrac{a_2}{|a_2-\overline{\kappa} a_1|}
\end{pmatrix}
\begin{pmatrix} \phi(-0) \\ \phi'(-0) \end{pmatrix}.
\end{equation}
\end{thm}

Note in these conditions that parameters   $a_1$, $a_2$ and $\kappa$ depend nonlinearly on  coupling constant $\beta$ as well as functions $f_1$, $f_2$, $a$ via  $\omega_\beta$; all elements of  the matrix are real, since $a_0$ and $a_2$ are real number.
The limit operator $\cT$ is self-adjoint, because the determinant of matrix in \eqref{CouplCondsT} is equal to $1$ (cf. \eqref{PIConditions}). In fact,
\begin{multline*}
|a_2-\overline{\kappa} a_1|^{-2}\, \det
\begin{pmatrix}
a_0|\kappa|^2-2{\rm Re}(\overline{\kappa} a_1)+a_2   &    |\kappa|^2\\
      a_0a_2-|a_1|^2                &      a_2
\end{pmatrix}\\=
|a_2-\overline{\kappa} a_1|^{-2}
\big(a_0a_2|\kappa|^2-2 a_2{\rm Re}(\overline{\kappa} a_1)+a_2^2
-a_0a_2|\kappa|^2+|\kappa|^2|a_1|^2
\big)\\=
|a_2-\overline{\kappa} a_1|^{-2}
\big(a_2^2-2 a_2{\rm Re}(\overline{\kappa} a_1)
+|\kappa|^2|a_1|^2
\big)
=|a_2-\overline{\kappa} a_1|^{-2}\,|a_2-\overline{\kappa} a_1|^2=1.
\end{multline*}
Though conditions \eqref{CouplCondsT} contain the full matrix, we can not  assert that it is possible to approximate any point interaction \eqref{PIConditions} by operators $\OprT$. For instance, such approximation does not exist for the point interactions \eqref{PIConditions} with matrices
\begin{equation*}
  \begin{pmatrix}
    c_{11} & 0 \\ c_{21} & c_{11}^{-1}
  \end{pmatrix},
\end{equation*}
where $c_{11}$ is different from $1$; if $\kappa=0$, then the matrix  in \eqref{CouplCondsT} has the unit diagonal.
Therefore Theorems \ref{Thrm1} and \ref{ThmT} are in some sense mutually complementary.

Remark also that for any pair of linearly independent functions $f_1$, $f_2$ satisfying the assumptions of the theorem there exists a wide class of potentials $U$ for which condition $a_2\neq \bar{\kappa} a_1$ holds.

In view of Theorems 1 and 2 in \cite{GolovatyIEOT2018} we can expect that there exist at least six essentially different cases of the limiting behaviour for $\OprT$ as $\eps\to 0$. However, in this paper we restrict ourselves to analyzing only the case that is described in Theorem~\ref{ThmT}. This case is the most interesting from the physical viewpoint.

\section{Zero-Energy Resonances and Half-Bound States}

We  show first that the set $\mathcal{R}(V)$ of all resonance coupling constants for operator $-\frac{d^2}{d x^2}+\alpha V$ is not empty and furthermore it is rich enough for any function $V$ of compact support.

\begin{lem}\label{LemmaRV}
  (i) For each measurable function $V$ of compact support,  the resonant set $\mathcal{R}(V)$ is a countable subset of the real line with one or two accumulation points at infinity.

  (ii) For each $\alpha\in \mathcal{R}(V)$, the corresponding half-bound state $v_\alpha$ is unique up to a scalar factor. Moreover both the limits
      \begin{equation}\label{LimitsValpha}
        v_\alpha^-=\lim\limits_{x\to-\infty}v_\alpha(x), \qquad
         v_\alpha^+=\lim\limits_{x\to+\infty}v_\alpha(x)
      \end{equation}
   exist and are different from zero.
\end{lem}

\begin{proof}
Without loss of generality we  assume that $\supp V\subset \cI$, where $\cI=(-1,1)$.
Then operator $-\frac{d^2}{d x^2}+\alpha V$ possesses a half-bound state if and only if the problem
\begin{equation}\label{HBSproblemV}
  -v''+\alpha Vv=0,\quad x\in \cI, \qquad v'(-1)=0, \quad v'(1)=0
\end{equation}
has a non-trivial solution.  In fact, a half-bound state $v_\alpha$ is constant outside  $\cI$ as a bounded solution of equation $v''=0$ and hence  $v_\alpha'(-1)=v_\alpha'(1)=0$. From this we also deduce that there exist  the limits \eqref{LimitsValpha}. Obviously we have  $v_\alpha^-=v_\alpha(-1)$ and $ v_\alpha^+=v_\alpha(1)$.
In addition, both the values $v_\alpha(-1)$ and $v_\alpha(1)$ are different from zero in view of  uniqueness  for the Cauchy problem, because $v_\alpha$ is a non-trivial solution.

Problem \eqref{HBSproblemV} can be regarded as a spectral problem with spectral parameter $\alpha$. If $V$ is a function of fixed sign, then
\eqref{HBSproblemV} is a standard Sturm-Liouville problem and $\mathcal{R}(V)$ coincides with the spectrum of a self-adjoint operator in weighted Lebesgue spaces $L_2(V,\cI)$.
Otherwise, we can interpret \eqref{HBSproblemV}   as the eigenvalue problem  with  indefinite weight function $V$; the problem can be associated with a self-adjoint non-negative operator $K$ in a Krein space \cite{GolovatyManko:2009, GolovatyMFAT:2012}.
In both the cases the spectra of  such operators are real and discrete with accumulation points at $-\infty$ or $+\infty$ only.  Moreover all nonzero eigenvalues are simple; for the case of the Krein space, $\alpha=0$ is generally semi-simple.
 The reader can also refers to \cite{IohvidovKreinLanger} for the details of the theory of self-adjoint operators in Krein spaces.
It follows from the simplicity of spectra that  half-bound state $v_\alpha$ is  unique up to a scalar factor.
\end{proof}

The set $\mathcal{R}(h_1,h_2)$ of coupling constants, for which the operator $B$ possesses the double zero-range resonance, is also rich for any pair of $h_1$ and $h_2$.
We set $m_k=\|\fpr{h}_k\|$ and $\tau=\la \fpr{h}_1, \fpr{h}_2\ra$.

\begin{lem}\label{LemmaHBS}
Assume that $h_1$,  $h_2$ are linearly independent functions  of zero mean. Then set $\mathcal{R}(h_1,h_2)$ of double zero-range resonance for  operator $B$  is  the circle
\begin{equation*}
  \mathcal{R}(h_1,h_2)=\{\beta\in \mathbb{C}\colon |\beta-\beta_0|=\rho \}
\end{equation*}
in the complex plane, where
\begin{equation*}
 \beta_0=\frac{\bar{\tau}}{m_1^2m_2^2-|\tau|^2},\qquad
 \rho =\frac{m_1m_2}{m_1^2m_2^2-|\tau|^2}.
\end{equation*}
In addition, if $\beta\in \mathcal{R}(h_1,h_2)$, then the constant function and function
 $$
 \omega_\beta=e^{i\arg(\beta^{-1}+\tau)}m_2 h_1^{(-2)}-m_1 h_2^{(-2)}
 $$
are two linearly independent half-bound states of  $B$.
\end{lem}
Note that circle $\mathcal{R}(h_1,h_2)$ is well defined for linearly independent  $h_1$ and $h_2$, because then the first antiderivatives $\fpr{h}_1$ and $\fpr{h}_2$ are also linearly independent, and $|\tau|<m_1m_2$ in view of the Cauchy-Schwartz inequality. For instance,
if functions $\fpr{h}_1$ and $\fpr{h}_2$ are orthonormal, then $\mathcal{R}(h_1,h_2)$ is a unit circle centered at the origin, since $m_1=m_2=1$ and $\tau=0$. If $\fpr{h}_1$ and $\fpr{h}_2$ are simply orthogonal, then
$\mathcal{R}(h_1,h_2)=\{\beta\in \mathbb{C}\colon |\beta|=m_1^{-1}m_2^{-1} \}$. In the case when $h_2=h_1+\eps g$ and $\eps$ is small, that is to say, the angle between $h_1$ and $h_2$ is small, the center $\beta_0$ is
far from the origin and the radius $\rho$ is large, because then the difference $m_1m_2-|\tau|$ is of order $\eps$.
\begin{proof}
We start with the observation that $v=1$ is obviously a solution of equation
 \begin{equation*}
 -v''+\bar{\beta}\,\la h_2, v\ra\, h_1+\beta\,\la h_1,v\ra\, h_2= 0,
\end{equation*}
since $h_k$ are functions with zero-mean values.
For the same reason, the second anti-derivatives $\spr{h}_k$ are bounded on the whole line.
Then regarding this equation as the ``non-homogeneous'' one
\begin{equation}\label{HBSEqnNH}
 v''=\bar{\beta}\,\la h_2,v\ra\,h_1+\beta\,\la h_1,v\ra\,h_2,
\end{equation}
we can look for another half-bound state in the form
$\omega=c_1\spr{h}_1+c_2\spr{h}_2$. We do not take into account solution $x$ of the homogeneous equation, because it is unbounded as $|x|\to\infty$.

Since $h_1$ and $h_2$ are linearly independent, substituting  $\omega$ into \eqref{HBSEqnNH} yields
\begin{equation}\label{SystFV}
  \begin{cases}
    \beta\,\la h_1,\spr{h}_1\ra\,c_1+(\beta\,\la h_1,\spr{h}_2\ra-1)\,c_2=0,\\
   ( \bar{\beta}\,\la h_2,\spr{h}_1\ra-1)\,c_1+ \bar{\beta}\,\la h_2,\spr{h}_2\ra \,c_2=0.
  \end{cases}
\end{equation}
Because $h_j$ has compact support, the scalar product $\la h_j,\spr{h}_k \ra$ is finite, even though antiderivative $\spr{h}_k$ does not belong to $L_2(\Real)$. In addition, the integrating by parts shows
$ \la h_j,\spr{h}_k\ra=-\la \fpr{h}_j,\fpr{h}_k\ra$. Then \eqref{SystFV} becomes
\begin{equation}\label{SystSV}
  \begin{cases}
    \hskip19pt\beta m_1^2\,c_1+(\beta \tau+1)\,c_2=0,\\
   (\overline{\beta\tau}+1)\,c_1+\hskip19pt \bar{\beta}m_2^2 \,c_2=0.
  \end{cases}
\end{equation}
This  system has a non-trivial solution $(c_1,c_2)$ if and only if
$|\beta|m_1m_2=|\beta \tau+1|$. The condition can be written as
$|\beta^{-1}+ \tau|=m_1m_2$.

Given $a\in \mathbb{C}$ and $r\in \Real$, we consider the  circle $\{z\in \mathbb{C}\colon |z-a|=r\}$. Suppose that $|a|<r$. The  mapping $z\mapsto z^{-1}$ is a bijection from this circle onto another one
\begin{equation*}
\left\{
z\in \mathbb{C}\colon\left|z+\frac{\bar{a}}{r^2-|a|^2}\right|=\frac{r}{r^2-|a|^2}
\right\},
\end{equation*}
as is easy to check. Therefore the resonance region $\mathcal{R}(h_1,h_2)$ arises as the image of the circle
$\{z\in \mathbb{C}\colon|z+ \tau|=m_1m_2\}$
under the transformation $z\mapsto z^{-1}$. Note that $|\tau|<m_1m_2$  by the Cauchy-Schwartz inequality.

If $\beta\in \mathcal{R}(h_1,h_2)$, then \eqref{SystSV} admits a nontrivial solution having the form
\begin{equation*}
c_1=e^{i\arg(\beta^{-1}+\tau)}m_2, \qquad c_2=-m_1.
\end{equation*}
In fact, substituting this solution into the first equation yields
\begin{multline*}
   \beta m_1^2\,c_1+(\beta \tau+1)\,c_2=
 \beta m_1^2 m_2 e^{i\arg(\beta^{-1}+\tau)} -m_1(\beta \tau+1)\\
 = \beta m_1 \,|\beta^{-1}+\tau|\, e^{i\arg(\beta^{-1}+\tau)} -m_1(\beta \tau+1)= \beta m_1 (\beta^{-1}+\tau) -m_1(\beta \tau+1)=0,
  \end{multline*}
since $m_1m_2=|\beta^{-1}+ \tau|$.
Hence $\omega_\beta=e^{i\arg(\beta^{-1}+\tau)}m_2 h_1^{(-2)}-m_1 h_2^{(-2)}$ is a half-bound state of  $B$, which is different from the constant one.
\end{proof}

\section{Proof of Theorems \ref{Thrm1} and \ref{ThmT} }
We start with some  assertions, which  will be used below.

\begin{lem}\label{LemmaStrongConv}
Let $\{S_\eps\}_{\eps>0}$ be a family of self-adjoint operators in a Hilbert space $\mathcal{L}$ and $\{W_\eps\}_{\eps>0}$ be a family of
unitary operators in $\mathcal{L}$.  Assume that $S_\eps\to S$ as $\eps\to 0$ in the norm resolvent sense, $W_\eps\to W$ in the strong operator topology as $\eps\to 0$ and $W$ is a unitary operator in $\mathcal{L}$.
Then the family of operators $Q_{\eps}=W_\eps S_\eps W_\eps^{-1}$ converges in the strong resolvent sense to the operator $Q=WSW^{-1}$ with the domain  $\{\phi\in \mathcal{L}\colon  W^{-1}\phi\in \dom S\}$.
\end{lem}

\begin{proof}   We first note that
\begin{multline*}
(Q_\eps-\zeta)^{-1}-(Q-\zeta)^{-1}
=W_\eps \bigl((S_\eps-\zeta)^{-1}-(S-\zeta)^{-1}\bigr)W_\eps^{-1}\\
+W_\eps(S-\zeta)^{-1}(W_\eps^{-1}-W^{-1})
+(W_\eps -W)(S-\zeta)^{-1}W^{-1},
\end{multline*}
provided $\zeta\in \mathbb{C}\setminus \mathbb{R}$.
The operator $S$ is self-adjoint as a limit of self-adjoint operators $S_\eps$ in the norm resolvent topology. From the last relation and the self-adjointness of $S$  we have
\begin{multline}\label{ResolventEst}
\|(Q_\eps-\zeta)^{-1}f-(Q-\zeta)^{-1}f\|\leq \|(S_\eps-\zeta)^{-1}-(S-\zeta)^{-1}\|\,\|f\|\\
+|\Im \zeta|^{-1}\|(W_\eps^{-1}-W^{-1})f\|+\|(W_\eps-W)(S-\zeta)^{-1}W^{-1}f\|
\end{multline}
for all $f\in \mathcal{L}$. The first term in the right-hand side tends to zero as $\eps\to 0$, since operators $S_\eps$ converge to $S$  in the norm resolvent sense. The last two terms are  infinitely small as $\eps\to 0$, in view of the strong convergence of $W_\eps$.
\end{proof}

We introduce two unitary operators
\begin{equation}\label{OprWeW}
    (W_\eps f)(x)=e^{i a\xe} f(x),\qquad
    (Wf)(x)=e^{i \mu H(x)} f(x),
\end{equation}
in $L_2(\Real)$, where $a$ and $\mu$  given by \eqref{FuncA} and \eqref{Sigma} respectively, and $H$ is the Heaviside step function
\begin{equation*}
  H(x)=
  \begin{cases}
    0, &\text{for \ }x<0,\\
    1, &\text{for \ }x>0.\\
  \end{cases}
\end{equation*}

\begin{lem}\label{LemmaConvergenceUh}
Let $W_\eps$ and $W$ be the unitary operators given by \eqref{OprWeW}.
Then  $W_\eps$ converge to  $W$ as $\eps\to 0$ in the strong operator topology.
\end{lem}
\begin{proof}
Without loss of generality we can assume that the support of the magnetic potential $A$ lies in $(-1,1)$.
Therefore $a(\eps^{-1}x)=0$ for $x<-\eps$ and $a(\eps^{-1}x)=\mu$ for $x>\eps$.
For each $f\in L_2(\Real)$ we have
\begin{multline}\label{ConvergenceUhS1}
        \|W_\eps f- Wf\|^2\leq\int_\Real\left|e^{i  a\xe}-e^{i \mu H(x)}\right|^2\,|f(x)|^2\,dx\\
        =\int_{-\eps}^\eps \left|e^{i a\xe}-e^{i  \mu H(x)}\right|^2\,|f(x)|^2\,dx\leq
        4 \int_{-\eps}^\eps |f(x)|^2\,dx,
    \end{multline}
since $a(\eps^{-1}x)=\mu H(x)$ for $|x|>\eps$.
The  right-hand side of \eqref{ConvergenceUhS1} tends to zero as $\eps\to 0$, by absolute continuity of the Lebesgue integral.
\end{proof}

\subsection{Proof of Theorem \ref{Thrm1}}
Let us consider the Schr\"{o}dinger operators
\begin{equation}\label{Seps}
    S_\eps= -\frac{d^2}{dx^2}+V_0(x)+\frac{\alpha}{\eps^2}\,V\left(\frac{x}{\eps}\right)
    +\frac{1}{\nu_\eps}\,U\left(\frac{x}{\nu_\eps}\right), \quad
\dom S_\eps=W_2^2(\mathbb{R}).
\end{equation}
It is of course that $S_\eps$ is a version of operator $\OprH$ given by \eqref{OprHenm} when the magnetic potential is disabled. We also denote by $S=S(\theta,\gamma)$ the  Schr\"odinger operator acting via $S\psi=-\psi''+V_0\psi$ on functions $\psi$ in~$\mathcal{V}$ obeying the  interface conditions
\begin{equation}\label{OprSConds}
\begin{pmatrix}
    \psi(+0) \\\psi'(+0)
\end{pmatrix}
=
\begin{pmatrix}
  \theta &  0\\
  \gamma &  \theta^{-1}
\end{pmatrix}
\begin{pmatrix}
    \psi(-0) \\ \psi'(-0)
\end{pmatrix}
\end{equation}
at the origin.
For all real $\theta$ and $\gamma$, this operator is  self-adjoint.

The proof of Theorem~\ref{Thrm1} is based on the  results obtained in \cite{GolovatyMFAT:2012} and \cite{GolovatyIEOT:2012}.
Let $\{\nu_\eps\}_{\eps>0}$ be a sequence  such that $\nu_\eps\to0$ as $\eps\to 0$ and the ratio $\nu_\eps/\eps$ tends to  $\lambda\in \Lambda$.
If the potential $\alpha V$ is resonant, then the operator family $S_\eps$ converges   in the norm resolvent sense as $\eps\to 0$ to  operator  $S=S(\theta(\alpha),\gamma(\alpha, \lambda))$,  where   $\theta$, $\gamma$ are given by \eqref{MapTheta}--\eqref{MapKappaInfty}.
We see at once that  operator $\OprHe$ is unitarily equivalent to  operator $S_\eps$, i.e.,  $\OprHe=W_\eps S_\eps W_\eps^{-1}$ with the unitary operator (the gauge transformation) $W_\eps$  given by \eqref{OprWeW} \cite{AvronHerbstSimon:1978}. For instance, it is easy to check that
\begin{equation*}
  -e^{i  a\xe}\frac{d^2}{dx^2}\left(e^{-i  a\xe}\phi(x)\right)=\Big(\rmi\,\frac{d}{dx}
    +\frac{1}{\eps}\,A\Big(\frac{x}{\eps}\Big)\Big)^2\phi(x),
\end{equation*}
since $a'=A$.
Next,   $W^{-1}(\dom \cH)\subset \dom S$, where $W^{-1}f=e^{-i \mu H}f$.
In fact, given $\phi\in \dom \cH$, we set $\psi=W^{-1}\phi=e^{-i \mu H}\phi$. Then we have $\psi(+0)=e^{-i \mu} \phi(+0)$, $\psi'(+0)=e^{-i \mu}\phi'(+0)$, $\psi(-0)=\phi(-0)$ and $\psi'(-0)=\phi'(-0)$.
Rewriting conditions \eqref{PointIntsH} for $\phi$ in the form
\begin{equation*}
\begin{pmatrix} e^{-i \mu}\phi(+0) \\ e^{-i \mu}\phi'(+0) \end{pmatrix}
=
\begin{pmatrix}
  \theta(\alpha) &  0\\
  \gamma(\alpha,\lambda) &  \theta(\alpha)^{-1}
\end{pmatrix}
\begin{pmatrix} \phi(-0) \\ \phi'(-0) \end{pmatrix},
\end{equation*}
we ascertain that $\psi$ satisfies \eqref{OprSConds} and therefore $\psi\in \dom S$. Obviously,
$$
W^{-1}\colon \dom \cH\to \dom S
$$
is a linear isomorphism. Therefore, the limit operator $\cH$ in Theorem~\ref{Thrm1} can be written as $\cH = WS W^{-1}$.

In view of Lemma~\ref{LemmaConvergenceUh}, the gauge transformations $W_\eps$ converge to $W$ in the strong operator topology.
Since the resolvents of $S_\eps$ converge to the resolvent of $S$
uniformly, we deduce from Lemma~\ref{LemmaStrongConv} that
\begin{equation*}
  \OprHe=W_\eps S_\eps W_\eps^{-1}\to WS W^{-1}=\cH \qquad\text{as } \eps\to 0
\end{equation*}
in the strong resolvent sense.

\subsection{Proof of Theorem \ref{ThmT}}
We can now argue almost exactly as in the proof of Theorem~\ref{Thrm1}.
First of all note that operators $\OprT$ given by \eqref{Teps} are  unitarily equivalent to  operators
\begin{equation}\label{Teps}
T_\eps=-\frac{d^2}{dx^2}+V_0(x)+\frac1{\eps^3}\,G_\eps+\frac1\eps\, U\Big(\frac{x}{\eps}\Big),
\end{equation}
namely $\OprT=W_\eps T_\eps W_\eps^{-1}$ with the gauge transformation $W_\eps$  given by \eqref{OprWeW}. Operator $G_\eps=G_\eps(g_1,g_2)$ is a  rank-two operator of the form
\begin{equation*}
  (G_\eps \psi)(x)=
  \bar{\beta}\,\la g_2(\eps^{-1}\,\cdot\,),\psi\ra\,g_1(\eps^{-1}x)
  +\beta\,\la g_1(\eps^{-1}\,\cdot\,),\psi\ra\,g_2(\eps^{-1}x),
\end{equation*}
where $g_1=e^{-i  a}f_1$ and $g_2=e^{-i  a}f_2$ are the same functions as in \eqref{GkNkP}.
In fact, a trivial verification shows that $F_\eps=W_\eps G_\eps W_\eps^{-1}$.

In Theorem~\ref{ThmT} we assumed $g_1$,  $g_2$ were linearly independent functions  of zero mean. Moreover $a_2\neq \bar{\kappa} a_1$. It has recently been proved in \cite{GolovatyIEOT2018} that if additionally
coupling constant $\beta$ belongs to the set $\mathcal{R}(g_1,g_2)$ of double zero-energy resonance for $B=-\frac{d^2}{dx^2}+\bar{\beta}\,\la g_2,\,\cdot\,\ra\,g_1+\beta\,\la g_1,\,\cdot\,\ra\,g_2$, then operators
$T_\eps$  converge as $\eps\to 0$ in the norm resolvent sense to  operator $T\psi=-\psi''+V_0\psi$ acting on functions $\psi\in\mathcal{V}$ obeying the  interface conditions
\begin{equation*}
\begin{pmatrix}
    \psi(+0) \\\psi'(+0)
\end{pmatrix}
=
e^{i \arg (a_2-\kappa \bar{a}_1)}\begin{pmatrix}
\dfrac{|\kappa|^2a_0-2{\rm Re}(\overline{\kappa} a_1)+a_2}
{|a_2-\overline{\kappa} a_1|}    &    \dfrac{|\kappa|^2}{|a_2-\overline{\kappa} a_1|}\\
      \dfrac{a_0a_2-|a_1|^2}{|a_2-\overline{\kappa} a_1|}  &      \dfrac{a_2}{|a_2-\overline{\kappa} a_1|}
  \end{pmatrix}
\begin{pmatrix}
    \psi(-0) \\ \psi'(-0)
\end{pmatrix}
\end{equation*}
at the origin.  Therefore $\cT=WTW^{-1}$,
by reasoning similar to that in the proof of Theorem~\ref{Thrm1}.
We can now repeatedly apply Lemma~\ref{LemmaStrongConv} for
operator families $\{T_\eps\}_{\eps>0}$ and  $\{W_\eps\}_{\eps>0}$
to deduce the strong resolvent convergence
\begin{equation*}
  \OprT=W_\eps T_\eps W_\eps^{-1}\to WTW^{-1}=\cT
\end{equation*}
as $\eps\to 0$.

\section{Final Remarks}

In Theorem~\ref{Thrm1} we obtained in the limit the coupling conditions
\begin{equation*}
\begin{pmatrix} \phi(+0) \\ \phi'(+0) \end{pmatrix}
=e^{i\mu(A)}\begin{pmatrix} \theta(V) & 0 \\ \gamma(V,U)  & \theta^{-1}(V) \end{pmatrix}
\begin{pmatrix} \phi(-0) \\ \phi'(-0)\end{pmatrix},
\end{equation*}
in which the magnetic potential $A$ appeared only in the phase factor $e^{i\mu(A)}$. This situation is typical for  potential perturbations of Schr\"{o}dinger operators.

Unlike the previous case, in which the potential perturbation was invariant with respect to the gauge transformation $W_\eps$, the finite-rank perturbation $F_\eps$ is not invariant. In fact,  $F_\eps=W_\eps G_\eps W_\eps^{-1}$; transformation $W_\eps$ rotates the plane ${\rm span}\{f_1,f_2\}$ when we change parameter $\eps$.
This is certainly the main reason why the magnetic field $A$ has an effect on all coefficients in the coupling conditions
\begin{equation*}
\begin{pmatrix} \phi(+0) \\ \phi'(+0) \end{pmatrix}
=e^{i\mu(A)}\begin{pmatrix} c_{11}(A) & c_{12}(A) \\ c_{21}(A) & c_{22}(A) \end{pmatrix}
\begin{pmatrix} \phi(-0) \\ \phi'(-0)\end{pmatrix}
\end{equation*}
appearing as the solvable model in Theorem~\ref{ThmT}. Of course, the coefficients $c_{kj}$  depend on potentials $V$, $U$ and functions $f_1$, $f_2$ too.

\end{document}